\documentclass[12pt]{amsart}
\usepackage{amsmath}
\usepackage{amssymb,amscd}
\usepackage[T1]{fontenc}
\usepackage[latin1]{inputenc}
\usepackage[colorlinks=true, urlcolor=blue,bookmarks=true,bookmarksopen=true, citecolor=blue,hypertex]{hyperref}

\numberwithin{equation}{section}

\newtheorem{defn}{Definition}[section]
\newtheorem{theorem}{Theorem}[section]

\newtheorem{corollary}[theorem]{Corollary}

\newtheorem{lemma}[theorem]{Lemma}
\newtheorem{proposition}[theorem]{Proposition}
\newtheorem{remark}[theorem]{Remark}

\def \begineq{\begin{equation}}
\def \endeq{\end{equation}}

\def \bb{\mathbb}

\def \CC{{\bb{C}}}

\def \QQ{{\bb{Q}}}
\def \RR{{\bb{R}}}

\def \ZZ{{\bb{Z}}}

\def \({\left(}
\def \){\right)}
\def \<{\langle}
\def \>{\rangle}
\def \bar{\overline}

\begin{document}

\title[Hodge structures orbifold Hodge numbers and a correspondence in quasitoric orbifolds]{ Hodge structures orbifold Hodge numbers and a correspondence in quasitoric orbifolds}
\author[Saibal Ganguli]{Saibal Ganguli}
\address{  Harish-Chandra Research Institute
Chhatnag Road, Jhusi
Allahabad 211 019
India}
\email{saibalgan@gmail.com}

\abstract
{ We  give  Hodge structures on quasitoric orbifolds. We define orbifold Hodge numbers and show a correspondence 
of orbifold Hodge numbers for crepant resolutions of quasitoric orbifolds.
In short we extend Hodge structures to a non complex non almost complex setting .}

\endabstract
\maketitle
\section{\bf{ Introduction } }
The purpose of this paper is to extend Hodge structures in a non complex non almost complex setting. First we give a
 canonical
Hodge structure (see \ref{can}) to quasitoric orbifolds. We compute  Hodge numbers. Then as an application we define orbifold Hodge
numbers and show a correspondence of these numbers for crepant resolutions. Since toric manifolds and possibly 
orbifolds are used widely in physics
as target spaces of  non linear sigma models, with  canonical Hodge  structures possible in quasitorics  
we feel  these spaces  can also be put into some consideration and inspection
by experts of the above  areas. Also this is an extension of Deligne's mixed Hodge structures in a non-complex 
algebraic setting
thus will draw interest of Mathematicians. During the process we also stumble on a short proof
of Hodge numbers of Projective toric orbifolds where its Hodge structure is Deligne's Hodge structures.
  
  Physicists believe that orbifold string theory is equivalent to ordinary string
theory of certain desingularizations. This belief motivated a body of conjectures,
collectively referred to as the orbifold string theory conjecture. The
conjecture we are interested in is the K-orbifold string theory conjecture. It
states that there is a natural isomorphism between the orbifold K-theory of
a Gorenstein orbifold and the ordinary K-theory of its crepant resolution.
To construct a natural isomorphism as the conjecture demands, is a very hard
problem. But many weaker versions of the conjecture that compare Euler
numbers, Hodge numbers, etc. has been studied extensively in the literature
in the case of algebraic orbifolds and orbifolds with global quotients. Batyrev
and Dais in \cite{[BT]}  Poddar and Lupercio  in \cite{[LP]} and Yasuda in \cite{[Yas]}, proved in particular the equality of orbifold Hodge numbers  of smooth crepant resolutions
for Gorenstein   algebraic 
orbifolds and for the non-Gorenstein algebraic case was proved by Yasuda  \cite{[Yas1]}. We generalize this correspondence to a non-algebraic non- analytic and non-global setting.
The correspondence also  
generalizes the string
theoritic Hodge number Correspondence of Batyrev and Dias to a non- algebraic non-analytic setting.

  A quasitoric orbifold (defined in \cite{[DJ]}, \cite{[PS]}) is a $2n$-dimensional differentiable orbifold equipped with
       a smooth action of the $n$-dimensional compact torus such that the orbit
       space
      is diffeomorphic as manifold with corners to a $n$-dimensional simple polytope (A $n$- dimensional polytope is called simple
      if every vertex is the intersection of exactly $n$ codimension one faces). The preimage of every codimension one face is a torus invariant
      $(2n-2)$-dimensional quasitoric orbifold
      which is stabilized by a circle subgroup of the form $\{(e^{2\pi a_{1}t}, \ldots e^{2\pi a_{n}t}): t\in \RR \}$.
       The vector $ (a_{1}, \ldots a_{n})$ is a primitive integral vector called the
characteristic
 vector associated with this  codimension one  face.  In general  a codimension $k$ face is the intersection of $k$  codimension one faces and its characteristic set consists of 
 the characteristic vectors of these codimension one faces. The characteristic set of every face is linearly independent over $\RR$.
   
   Quasitoric manifolds (and orbifolds, although not in full generality,) were introduced
in \cite{[DJ]}  . They got their present name in \cite{[BP]}. They are generalizations of
smooth projective toric varieties. They include manifolds which do not admit an
almost complex structure such as $CP^{2}\sharp CP^{2}$. A broader class of quasitoric orbifolds
were defined in \cite{[PS]}. In that paper and the subsequent papers \cite{[GP]} and \cite{[GP2]}, questions
relating to homology, cohomology, almost complex structures and equivariant blowdown
maps were addressed. McKay correspondence of Chen-Ruan cohomology was also established for four
and six dimensional quasitoric orbifolds in \cite{[GP]} and \cite{[GP2]}.
   
    Thus quasitoric orbifolds   give a vast family of orbifolds which are neither complex or almost complex.
  In this paper we address the question of Hodge structures of these orbifolds. Since $CP^{2}\sharp CP^{2}$ is a non- almost complex quasitoric orbifold it gets a Hodge structure in spite of being
   non-almost complex. This opens possibility of Hodge theory in a non almost complex setting. The idea of the proof is to relate the  cohomology of a  quasitoric orbifold
  with  the cohomology of a projectiove toric orbifold which is a Kahler orbifold . The  complex De Rham cohomology of the two spaces are shown to be isomorphic as a graded vector
  space and so
  the Hodge structure of one can be pulled back to the other.
     
     With Hodge structures possible we can define orbifold Hodge numbers and get a correspondence for these numbers under crepant blowdown by imitating Batyrev-Dias correspondence.
    
    A general quasitoric orbifold differs combinatorially from a projective toric orbifold
in the following manner. In the case of a projective toric orbifold, the characteristic
vector of a codimension one face of the orbit polytope is normal to that face. This
enables the characteristic vectors to generate cones that fit together to form a toric
fan. In a general quasitoric orbifold the normality condition is relaxed to linear independence of characteristic sets of faces. 

    The paper is organized in the following manner. In section $2$ we give a combinatorial construction of  quasitoric orbifolds. In section $3$ we discuss Betti numbers of quasitoric
   orbifolds. In section $4$ we define Hodge structures and provide a canonical Hodge structure by showing cohomological vector space isomorphism between a quasitoric and a projective
    toric. In section $5$ as an application we define and show orbifold Hodge number correspondence. In section 6 we give an argument to show every simple polytope is combinatorially
    equivalent to a rational polytope which is used to show cohomological isomorphism of a quasitoric with a projective toric.
\section{\bf{Quasitoric orbifolds } }\label{smooth}
 In this section we describe the combinatorial construction of quasitoric orbifolds. Notations established in this section will
be used later.

 Take a copy $ N$of $ \ZZ^{n}$ and form a torus  $T_N := (N \otimes_{\ZZ} \RR) / N \cong \RR^n/ N $.

 Take a submodule  $M$ of $ N$ of rank $m$ and  construct the torus  $T_M := (M \otimes_{\ZZ} \RR) / M $ of 
dimension $m$.
Define the map $\zeta_M: T_M \to T_N$ the obvious  map generated by the inclusion map $M \to N$.
\begin{defn}\label{tlambda}
 We define the image of $T_M$ under the map $\zeta_M$ as $T(M)$. If M is a sub-module of rank 1 and $\lambda$ is 
the generator then we call the image $T(\lambda)$.

\end{defn}
\begin{defn}\label{pol}
 A polytope is $P$ is a subset of $\RR^{n}$  which is diffeomorphic as manifolds with corners to a convex hull
 $C$  of a finite number of points in $\RR^{n}$. The faces of $P$ are images of faces of $C$.
\end{defn}
\begin{defn}\label{simpol}
 A simple polytope  is a polytope  where each vertex is an intersection of n co-dimension one faces which are in general
position.
\end{defn}
\begin{defn}\label{facet}
Codimension one  faces of a polytope $P$  are called facets. In a simple polytope every k dimensional face is an
intersection of n-k  facets. We call $\mathcal{F}=\{F_1,F_2 \ldots F_M\}$ the set of facets of the simple polytope P.
\end{defn}
\begin{defn}\label{charac}
   We define a map $\Lambda:\mathcal{F} \to \ZZ^{n} $ where $ F_i$ is mapped to $\Lambda(F_i)$ and  if $F_{i_1}\ldots F_{i_k}$
 intersect to form a face of the polytope $P$ then the corresponding $\Lambda(F_{i_1}) \ldots \Lambda(F_{i_k})$ are linearly
independent. From now onwards we call $\Lambda(F_i)$ as $\lambda_i$ and call it a characteristic vector and $\Lambda$ the characteristic function.
\end{defn}
\begin{remark}
In this article we consider only primitive characteristic vectors and call the  corresponding quasitoric orbifolds
as primitive quasitoric orbifolds. The codimension of  the singular locus of these  orbifolds is at least 4.
\end{remark}
\begin{defn}
 For a face $F$ define $\mathcal{I}(F)=\{i: F \subset F_i ,F_i \in \mathcal{F} \}$ . The set $\Lambda_F =\{\lambda_i : i \in \mathcal{I}(F)\}$  
is called the characteristic set of F. We call $N(F)$ be the sub module generated $\Lambda_F$. If $\mathcal{I}(F)$is
 empty $N(F)=0$ .
\end{defn}

  For any point $p$ in the polytope we denote  $F(p)$ the face whose relative interior contains $p$. We define an
equivalence relation in $P \times T_N $ where $(p,t_1)\sim (q,t_2)$  if $p=q$  and ${t_2}^{-1}t_1 \in T(N(F(p))$ 
where
$N(F(p))$ is the sub module of $N$ generated by  integral linear combinations of vectors of $\Lambda_{F(p)}$.
 The quotient space $ X = P \times T_N / \sim$  has a structure of an $2n$ dimensional orbifold and are
called quasitoric orbifolds.
   
   The pair $(P,\Lambda)$ is a model for the above space. If  vectors comprising  $\Lambda_F$ are
unimodular for all faces $F$ we get a quasitoric manifold. The  $T_N$ action on $P \times T_N $ induces a torus 
action
on the quotient space $X$, of the equivalence relation, and quotient of this action is the polytope $P$.
Let us  denote the  quotient map by  $\pi:X \to P$. ${\pi}^{-1}(w)$ for a vertex $w$ of $P$ is a fixed point of
the above action and we will  denote it by $w$  without confusion.
\subsection{Orbifold structure}
 For every  vertex $w$ in $P$ consider open set $U_w$ of  $P$ the complement  of all faces not containing the 
vertex $w$.
We define 
\begin{equation}
  X_w=\pi^{-1}(U_w)= U_w \times T_N / \sim .
 \end{equation}
 For any face $F$ containing the vertex $w$ there is a natural inclusion of $N(F)$ in $ N(w)$ and $T_{N(F)}$ in
 $ T_{N(w)}$.
We define another equivalence relation $\sim_w$ on $ U_w \times T_{N(w)}$ as follows.

   For $ p \in U_w$, let $F$ be the the face which contains $p$ in its relative interior, by definition $F$ 
contains $w$.
We define the relation as $(p,t_1) \sim_w (q,t_2)$,if $p=q$ and ${t_2}^{-1} t_1 \in T_{N(F)} $.
We define
\begin{equation}
 \tilde{X_w} =U_w \times  T_{N(w)} / \sim_w .
\end{equation}
  The above space is equivariantly diffeomorphic to an open set in $ C^{n}$ with the standard torus action on 
$C^{n}$and
$T_{N(w)}$ action on $\tilde{X_w}$. The diffeomorphism will be clear from the subsequent discussion.
    The map $\zeta_{N(w)}: T_{N(w)} \to T_N$ induces a map from $\zeta_w :\tilde{X_w} \to X_w $ in the following
way
\begin{equation}
 \zeta_w((p,t)\sim_w)=(p,\zeta_{N(w)}(t))\sim .
\end{equation}
The kernel of the map $\zeta_{N(w)}$ is $G_w=N/N(w)$ is a subgroup $T_{N(w)}$ and has a smooth  action on
 $\tilde{X_w}$ and the quotient
of this action is $X_w$. This action is not free and so $X_w$ is an orbifold and the uniformizing  chart of
$X_w$ is $(\tilde{X_w},G_w,\zeta_w)$.

   We define a homeomorphism $\phi(w):\tilde{X_w} \to \RR ^{2n} $ as follows. Assume  without loss of generality
 $F_1, F_2 \ldots F_n$ are the facets containing $w$ and $p_i(w)=0$ is the the facet  $F_i$ and in
$U_w$ $p_i^{,}s$ have non-negative values with  positive in interiors of $U_w$. Let $\Lambda_w$ be the corresponding
set of characteristic vectors represented as follows
\begin{equation}
 \Lambda_w= [\lambda_1 \ldots \lambda_n] .
\end{equation}
If $q(w)$ be the representation of the angular coordinates of $T_N$ in  the basis with respect to 
$\lambda_1 \ldots \lambda_n$ of $N \otimes_{\ZZ} \RR $. Then the standard coordinates $q$ are related in the following
manner to $q(w)$
\begin{equation}
 q= \Lambda_w q(w) .
\end{equation}
 The homeomorphism $\phi(w):\tilde{X_w} \to \RR ^{2n} $ is
\begin{equation}
x_i = x_i(w):= \sqrt{p_i (w)} \cos(2 \pi q_i(w) ), \quad
      y_i = y_i(w):= \sqrt{p_i(w) } \sin( 2 \pi q_i(w) ) \quad {\rm for}\;
      i=1,\ldots,n .
\end{equation}
We write
\begin{equation}\label{cplxcoor}
 z_i = x_i + \sqrt{-1} y_i, \quad {\rm and} \quad
z_i(w) = x_i(w) + \sqrt{-1}y_i(w).
\end{equation}

 Now consider the action of $G_w = N/N(w)$ on $\widetilde{X}_w$. An element
     $g$ of $G_w$ is represented by a vector $\sum_{i=1}^n a_i \lambda_i $ in
   $N$ where each $a_i \in  \QQ$.  The action of $g$ transforms the coordinates
   $q_i(w)$ to $q_i(w) + a_i$. Therefore
    \begin{equation}\label{action}
    g\cdot (z_1,\ldots, z_n) = (e^{2\pi \sqrt{-1} a_1} z_1,\ldots, e^{2\pi \sqrt{-1}a_n} z_n).
     \end{equation}
We define
\begin{equation}\label{gx3}
  G_F := ((N(F)\otimes_{\ZZ} \QQ) \cap N) / N(F).
\end{equation}
We denote the space X with the above orbifold structure by
$\bf{X}$.
\subsection{\bf{Invaraint Suborbifolds}} 
The $T_N$  invariant subset $\pi^{-1}(F)$  where  $F$ is a  face of $P$ is a  quasitoric orbifold. The face
$F$ acts as the polytope of $\bf{X}(F)$ and the  characteristic vectors are obtained by  projecting
characteristic vectors of $\bf{X}$ to $N/\tilde{N(F)}$ where $\tilde{N(F)}=N(F)\otimes_{\ZZ} \QQ \cap N$.
With this structure $\bf{X}(F)$ is a suborbifold of $\bf{X}$. The suborbifolds corresponding to the facets are called
characteristic suborbifolds. We denote the interior of a face by $ F^{\circ}$. The interior of a vertex $w^{\circ}$
is $w$.

\subsection{\bf{Orientation}}
Quasitoric orbifolds are oriented. For more detailed discussion see  section  2.8 of \cite{[GP2]}. A choice of 
orientation
 of $T_N$  and a  choice of orientation of the polytope $P$ induces an orientation of the  quasitoric orbifold $\bf{X}$.
\subsection{\bf{Omniorientation}}
 A choice of  orientations of the normal bundles of the orbifolds corresponding to the facets (which we named as characteristic suborbifolds)
is termed as fixing an omniorientation. This is equivalent to fixing the sign of the characteristic vector
 associated to the facet (note:we call co-dimension one faces as facets). A quasitoric orbifold with  a fixed omniorientation
 is called and {\bf omnioriented quasitoric orbifold}. A quasitoric orbifold is    {\bf positively omnioriented} 
if it has an omniorientation such that for every
vertex $w$ ,$\Lambda_w$ has a positive determinant. For more detailed discussion see  %
section  2.9  of \cite{[GP2]}.%
\section{\bf{Betti numbers of a Quasitoric Orbifold}}
Poddar and Sarkar computed the $\QQ$ homology and cohomology of  quasitoric orbifolds  in \cite{[PS]}. In particular
 the computation of homology in Section $4$ of \cite{[PS]}   gives a strong connection between the combinatorics of
the polytope $P$ and the Betti numbers. We discuss the connection in following  proposition.
\begin{proposition}
 Quasitoric orbifolds with  combinatorially equivalent polytopes have same Betti numbers .
\end{proposition}
\begin{proof}
 A  brief discussion of the homology computation in  \cite{[PS]}  is required to establish the above proposition.
The computation depends on defining a continuous height function on the polytope $P$ with following properties.
\begin{enumerate}{\label{list}}
 \item Distinguishes vertices.
\item Strictly  increases or decreases on  edges. 
\item Each face has a unique maximum and minimum vertex.
\item The maximum vertex is the unique vertex of the face where all the edges of the face meeting the vertex has 
a maximum on the vertex.
\item The minimum  vertex is the unique vertex of the face where all the edges of the face meeting  the vertex 
has a minimum on the vertex.
\end{enumerate}
 A vertex distinguishing  linear functional  of $\RR^{n}$  does the job. Here we assume $P$ is embedded in  $\RR^{n}$.
Once we have such a function we orient the edges of the polytope in  increasing direction of the height function and
 arrange the {\bf vertices in increasing order of height}.
We define  index $i_w$ of  a vertex $w$ as the number of incoming edges. The smallest face containing these 
incoming edges is  the largest 
face $F_w$ which
has $w$ as the  maximum vertex. Now start attaching $2i_w$ $q$-cells  following the increasing order 
of vertices. The $q-$ cell
covers the  entire  inverse image of $F_w$ in the orbifold. For definition and  description of $q$-cells and 
the attaching maps we ask the  {\bf reader to consult} 
 \cite{[PS]}.

   Now each face  has a unique maximum vertex $w$ and interior of the face will be contained in $F_w$ by
 points
$3$ and $4$ above. So each face gets covered and each point  in the orbifold is in the  interior of exactly 
 one $q$-cell. Considering 
the polytope as a face there will be exactly one 0 $q$-cell and
one $2n$ $q$-cell. Thus we get a $q$ cellular decomposition of the quasitoric orbifold.

   Now it is shown in \cite{[PS]} that the $2k$ Betti numbers   depends on  the number of vertices with index 
$k$ while
the odd Betti numbers are zero. Now if we have two quasitoric orbifolds with two combinatorially equivalent 
polytopes(which means they are diffeomorphic as manifold with  corners) the height function of one
composed with the diffeomorphism gives a height function of the other with identical vertex indices. Thus
their Betti numbers will be same by what is done in \cite{[PS]}.
\end{proof}
\begin{corollary}\label{iden}
 The  dimension of each degree of $\QQ,\RR$ and $\CC$  singular  cohomology of a
quasitoric orbifolds $\bf{X}$ and $\bf{X^{'}}$ with combinatorially equivalent polytopes are same .
\end{corollary}
\begin{proof}
 By Universal coefficient theorem.
\end{proof}
\begin{corollary}\label{sim}
 The  dimension of each degree of $\QQ,\RR$ and $\CC$  singular  cohomology of a
quasitoric orbifold $\bf{X}$  is identical with a projective toric orbifold $X^{'}$ .
\end{corollary}
\begin{proof}
  Take a  quasitoric orbifold ${\bf X}$.  A slight perturbation makes the polytope $P$ associated with 
the orbifold into a rational polytope (see section 5.1.3.in \cite{[BP]} or see appendix) without  changing its combinatorial class,
  and with suitable dilations makes it into an
 integral polytope $P^{'}$ which is combinatorially equivalent to $P$. Now from $P^{'}$ taking normal  fan we get a projective
toric orbifold $X{'}$(the analytic structure determines the
 orbifold structure so we do not use the bold notation) with polytope $P^{'}$.
Since polytope $P$ and polytope $P^{'}$ are combinatorially equivalent by \eqref{iden} the above holds.
\end{proof}
\begin{corollary}\label{iso}
 Each degree of the cohomology of the two spaces are isomorphic.
\end{corollary}
\begin{proof}
Since they have the same dimension  and the dimensions are finite so the vector spaces are isomorphic. We define 
the isomorphisms as
$J_k$ where $k$ is the degree of the cohomology.
\end{proof}
\section{\bf{Hodge Structure}}
\begin{defn}
 A pure Hodge structure of weight n consists of an Abelian group $H_K$ and a decomposition of its complexification
into  complex subspaces $H^{p,q}$ where $p+q=n$ with the property  conjugate of $H^{p,q}$ is $H^{q,p}$.
\begin{equation}
 H_C = H_K \otimes_{\ZZ} \CC = \bigoplus_{p+q=n} H^{p,q} .
\end{equation}
and
\begin{equation}
\overline{H^{p,q}} = H^{q,p} .
\end{equation}
\end{defn}
\begin{defn}
  By a  Hodge structure on a compact space we mean the  singular cohomology group of degree $k$ has a pure
Hodge structure of weight $k$ for all $k$.
\end{defn}

\begin{proposition}\label{kahler}
 Kahler compact orbifolds have a canonical Hodge structure.
\end{proposition}
\begin{proof}
 By Baily's Hodge decomposition see \cite{[WB]}.
\end{proof}

\begin{proposition}
 Projective toric orbifolds coming from integral simple polytopes are Kahler.
\end{proposition}
\begin{proof}
 By theorem 8.1  9.1 and 9.2  in \cite{[LT]}.
\end{proof}

\begin{corollary}
 Projective toric orbifolds coming from integral simple polytopes have a canonical Hodge structure.
\end{corollary}
\begin{defn}
 Let $H^{p,q}$ be the $(p,q)$ Hodge component  of the canonical Hodge structure on 
Projective toric orbifolds $X^{'}$  coming from integral  simple polytopes.
We define 
\begin{equation}
 H^{p,q}(X^{'})=H^{p,q}.
\end{equation}
and
\begin{equation}
 h^{p,q}(X^{'})=dim(H^{p,q}(X^{'})).
\end{equation}
\end{defn}

  Let $\bf{X}$ be a  quasitoric orbifold and $X^{'}$ be the projective toric orbifold whose integral simple polytope $P^{'}$
 is combinatorially  equivalent to the polytope $P$ of $\bf{X}$. We assign

\begin{equation}
 H^{p,q}({\bf X})=J_k(H^{p,q}(X^{'})). 
\end{equation}
where $p+q=k$ and $ J_k$ is the isomorphisms of the degree $k$ cohomologies
 defined in \eqref{iso}.

\begin{theorem}
 The above assignment defines a Hodge structure on $\bf{X}$ depending on $J_k$. For independence of $J_k$ see \ref{can}.
\end{theorem}
\begin{proof}
 By above and corollary \eqref{sim}.
\end{proof}
\begin{theorem}
 The assignment does not depend on $X^{'}$.
\end{theorem}
\begin{proof}
 To show the above we must understand the $E$-polynomial.
Let $Y$ be an algebraic variety  over $ \mathbb{C}$ which is not necessarily compact
or smooth. Denote by $ h^{p,q}(H_c^{k}(Y))$ the dimension of the $(p,q)$ Hodge component
of the $k$-th cohomology with compact supports.  This is a generalization of the Hodge structures discussed on 
the above  class of   compact projective toric orbifolds and are called mixed Hodge structures. For more detailed 
discussion we ask the reader to consult \cite{[CS]}.
 
    We define
\begin{equation}
  e^{p,q}(Y)=\Sigma _{k \ge 0 } (-1)^{k} h^{p,q}(H_c^{k}(Y)).
  \end{equation}
 The polynomial
 \begin{equation}
 E(Y; u, v) := \Sigma _{p,q}  e^{p,q}(Y) u^{p}v^{q}
\end{equation}
is called $E$-polynomial of $Y$.
When we have a proper Hodge structure like the above class of  compact projective  toric orbifold $X^{'}$ ,  

\begin{equation}\label{pure}
  e^{p,q}(X^{'})= (-1)^{p+q} h^{p,q}(X^{'}).
  \end{equation}

Now if we have a stratification of  an  algebraic variety $Y$  by disjoint  locally  closed sub-varieties $Y_i$
(i.e $Y_i \subset Y$ 
and $ Y=\cup _i Y_i$) by proposition $3.4$ of \cite{[BT]}
\begin{equation}\label{sum}
 E(Y;u,v)= \Sigma_i E(Y_i;u,v)
\end{equation}
 Now in a projective toric orbifold coming from a integral  simple  polytope as in our case 
we have a stratification by  
algebraic tori corresponding
 to the interior of each face. Let $X^{'}$ be the concerned orbifold and $F_i$ be a $k$ dimensional face of the 
corresponding polytope
then $\pi^{-1}(F_i^{\circ})$ is a  $k$ dimensional algebraic tori  which we denote $X^{'}_i$. So 
by \eqref{sum} we have
\begin{equation}\label{newsum}
 E(X^{'};u,v)= \Sigma_i E(X^{'}_i;u,v).
\end{equation}
  {\bf Here $i$ runs over all the faces}.  
Now if we have two projective toric orbifolds $X^{'}$ and $X^{''}$ both having combiantorially equivalent
 polytopes
 with that of $\bf{X}$ by \eqref{newsum} we claim they have the same $E$-polynomial. This is because since they
 have combinatorially
equivalent polytopes, number of faces of a  given dimension will be same for each polytope. So the sum  on 
the right
hand side of \eqref{newsum} can be partitioned into  $E$-polynomial of $k$ dimensional algebraic tori with  a multiplicity of 
number of faces of 
 dimension $k$, where $k$ runs from 0 to dimension of the polytopes. Since same dimensional algebraic tori have same $E$-polynomial
the above claim holds.

    Thus the Hodge numbers of the two projective toric  orbifolds will be same by \eqref{pure}. So the theorem 
holds.
\end{proof}
\begin{theorem}
 The Hodge numbers of a quasitoric orbifold  are as follows
 
$h^{p,q}({\bf X})=0$ if $ p \neq q$ and $h^{p,p}({\bf X})=dim(H^{2p}({\bf X},\CC))$. 
\end{theorem}
\begin{proof}
 We show this for projective toric orbifolds coming from integral  simple polytope. We know that the $E$-polynomial of a $k$-dimensional
algebraic torus is $(uv-1)^k$. Since by \eqref{newsum} the $E$- polynomial of the projective toric orbifold decomposes
into sum of $E$-polynomial of algebraic tori and  since $E$-polynomial of the algebraic tori 
have only  terms of the form $(uv)^{l}$ 
, 
implies  that
coefficient of $u ^{p}v^{q}$ is zero if $p \neq$ q in the  $E$- polynomial of the projective toric orbifold. Since these projective toric orbifolds have a
proper Hodge structure the  claim
of the theorem is true. 
\end{proof}
\begin{theorem}\label{can}
 The above Hodge structure does not depend on $J_k$ and is canonical.
\end{theorem}
\begin{proof}
 Since there is only one non-zero Hodge number for a given degree of cohomology, different $J_k$ will define the same Hodge decomposition.
\end{proof}
\begin{remark}
 The above proof of Hodge numbers of quasitoric orbifolds is also a proof for Hodge numbers of projective toric orbifolds coming from Deligne's mixed hodge structures. We have not seen
 this proof in literature before.
\end{remark}

\subsection{Example}\label{exam}
We compute the Hodge structure for $CP^{2} \sharp CP^{2}$  which does not have an almost complex structure. We take
 a projective toric orbifold  $X^{'}$ with a combinatorially equivalent polytope $ P^{'}$. Since the polytope
of $P$  is a four sided polygon (see example 1.19 \cite{[DJ]}) it will have four vertices, four edges and one $2$-face.
\begin{equation}
 E(X^{'};u,v)= (uv-1)^{2} + 4(uv-1) +4.
\end{equation}
\begin{equation}
 E(X^{'};u,v)=  u^{2}v^{2} + 2uv + 1.
\end{equation}
 This tallies with the cohomology of $CP^{2}\sharp CP^{2}$ and so we have the decomposition $h^{2,2}=1$ ,$h^{1,1}=2$ and
$h^{0,0}=1$,
\section{{\bf An application -Orbifold Hodge numbers and a correspondence}}
\subsection{Orbifold Hodge numbers}\label{orb_hog}
 Orbifold Hodge numbers for  closed global quotient orbifold  was defined in \cite{[EZ]} and \cite{[BT]} and for Kahler orbifolds in \cite{[Po1]}. They are the dimensions of the 
 Dolbeaut  orbifold  cohomology
 (see \cite{[Po1]} section 2.2). They depend on the twisted sectors
of the orbifold. The twisted sector for toric variety was computed in \cite{[Po]}. The determination
of twisted sectors of quaitoric orbifolds are similar in essence.
    Let $\bf{X}$ be an omnioriented quasitoric orbifold (i.e the signs of  characteristic vectors are fixed). Consider an 
element
$g$ belonging to to the group $G_F$ defined in equation \eqref{gx3}. Then ${g}$ may be represented by the vector
   $\sum_{j \in \mathcal{I}(F)} a_j \lambda_j $ where $ a_j$ is restricted to $[0,1)\cap \QQ$ and $\lambda_j$ belongs to
the characteristic set of $F$. We define the degree
shifting number or {\bf age} as
\begin{equation}\label{age}
\iota(g)= \sum a_j.
\end{equation}

For faces $F$ and $H$ of $P$ we write $F \le H$ if $F$ is a
sub-face of $H$, and $F < H$ if it is a proper sub-face. If $F \le H
$ we have a natural inclusion of $G_H$ into $G_F$ induced by the
inclusion of $N(H)$ into $N(F)$. Therefore we may regard $G_H$ as
a subgroup of $G_F$. Define the set
\begin{equation}
G_F^{\circ} = G_F - \bigcup_{F < H} G_H.
\end{equation}
Note that $G_F^{\circ} = \{ \sum_{j \in \mathcal{I}(F)} a_j \lambda_j |
0 < a_j < 1 \} \cap N  $, and $G_P^{\circ}= G_P =\{0\}$.
\begin{defn}
 We define the    $orbifold$ $dolbeault$ $cohomology$  groups 
 of an omnioriented quasitoric orbifold ${\bf X}$ to be
$$ H^{p,q}_{orb}({\bf X} ) =
\bigoplus_{F \le P} \bigoplus_{ g\in G_F^{\circ}} H^{p - \iota(g),
q - \iota(g)} (X(F)).$$
 Here $H^{p - \iota(g),q - \iota(g}(X(F))$ refers to the components of the Hodge structures defined above ,
  when $X(F)$ is considered
  as a quasitoric orbifold  ${\bf X}(F)$.
 The pairs $(X(F), g)$ where $F<P$  and $ g\in G_F^{\circ}$ are
called twisted sectors of ${\bf X}$. The pair $(X(P),1)$, i.e. the
underlying space $X$, is called the untwisted sector.
\end{defn}
\begin{defn}
We define $ orbifold$ $Hodge$ $numbers$ as $h^{p,q}_{orb}({\bf X})=dim(H^{p,q}_{orb}({\bf X}) )$.
\end{defn}
Now  we introduce some notation. Consider a co-dimension $k$ face
$F= F_1 \cap \ldots \cap F_k$ of $P$ where $k \ge 1$.
  Define a $k$-dimensional cone $C_F$ in $N\otimes \RR$ as follows,
\begin{equation}\label{cf}
C_F = \{ \sum_{j=1}^k a_j \lambda_j: a_j \ge 0 \}.
\end{equation}
 The group $G_F$ can be identified with the subset $Box_F $ of
 $C_F$, where
\begin{equation}\label{boxf}
 Box_F := \{
\sum_{j=1}^k a_j \lambda_j: 0\le a_j <1 \} \cap N.
\end{equation}
Consequently the set $ G_F^{\circ}$ is identified with the subset
\begin{equation}\label{boxfo}
 Box_F^{\circ} :=  \{ \sum_{j=1}^k a_j
\lambda_j: 0 < a_j < 1 \} \cap N .\end{equation} of the interior of
$C_F$. We define $Box_P = Box_P^{\circ}= \{0\} $.

 Suppose $w=F_1\cap \ldots \cap F_n$ is a vertex of $P$. Then
$Box_w =  \bigsqcup_{w\le F} Box_{F}^{\circ} $. This implies
\begin{equation}\label{Gdecom}
G_w = \bigsqcup_{w \in F} G_F^{\circ}.
\end{equation}
An almost complex orbifold is $SL$ if the linearization of each
$g$ is in $SL(n,\CC)$. This is equivalent to $\iota(g)$ being
integral for every twisted sector.
 Therefore, to suit our purposes, we make the following definition.

\begin{defn}\label{quasisl}
 An omnioriented  quasitoric
orbifold is said to be quasi-$SL$ if the age of every twisted
sector is an integer.
\end{defn}
\subsection{Blowdowns}
In order to get a blow up of a face we  replace a face by a facet with a new characteristic vector. Suppose $F$ is a face of $P$. We
 choose  a hyperplane $H = \{ \widehat{p}_0 = 0 \}$ such that $\widehat{p}_0$
 is negative on $F$ and $\widehat{P}:=\{\widehat{p}_0 > 0\} \cap P$ is a simple
 polytope having one more facet than $P$. Suppose $F_1, \ldots, F_m$ are the
 facets of $P$. Denote the facets $F_i \cap \widehat{P} $ by $F_i$
 without confusion. Denote the extra facet $H \cap P$ by
 $F_{0}$.

Without loss of generality let $F = \bigcap_{j= 1}^k F_j$.
  Suppose there exists a primitive vector
  $\lambda_{0} \in N$ such that
  \begin{equation}
 \lambda_{0} = \sum_{j= 1}^k b_j \lambda_j, \; b_j > 0 \, \forall
 \, j.
  \end{equation}
 Then the assignment $F_{0} \mapsto \lambda_{0}$
 extends the characteristic function of $P$ to a characteristic function
 $\widehat{\Lambda}$ on $\widehat{P}$. Denote the omnioriented quasitoric
 orbifold derived from the model $(\widehat{P}, \widehat{\Lambda}  )$ by
 ${\bf Y}$.
\begin{defn}
We define  blowdown a map ${\bf Y} \mapsto {\bf X}$  which is inverse of a blow-up. Such maps
 have been constructed in   \cite{[GP2]}.
\end{defn}

\begin{lemma} (Lemma 4.2  \cite{[GP2]}) If ${\bf X}$ is
positively omnioriented, then so is a blowup ${\bf Y}$.
\end{lemma}

\begin{defn}\label{crepant}
A blowdown or blow up is called crepant if $\sum b_j = 1 $.
\end{defn}

\begin{lemma} (Lemma 8.2  \cite{[GP2]}) The crepant blowup of a quasi-$SL$ quasitoric orbifold is
quasi-$SL$.
\end{lemma}
\subsection{Correspondence of orbifold  Hodge numbers}
 The statement
of the theorem we are going to prove is as follows
 \begin{theorem}
  For crepant blowdowns(or blowups) orbifold  Hodge numbers of   quasi-$SL$ quasitoric orbifolds do not 
change.
 \end{theorem}

\begin{corollary}
 For crepant resolution orbifold  Hodge numbers  of  quasi-$SL$ quasitoric orbifolds  do not change.
\end{corollary}

We admit the proof is similar to the proof of Mckay Correspondence of Betti- numbers of Chen-Ruan cohomology in the author's previous paper
\cite{[Ga]} and motivated by Strong Mckay Correspondence proof  \cite{[BT]} ,but still we give a detailed argument for
the convenience of the reader.
\subsection{Singularity and lattice polyhedron}
Following the discussion in Section $\ref{orb_hog}$ , a singularity of a
face F is defined by a cone $ C_{F}$ formed by positive linear
combinations of vectors in its characteristic  set ${\lambda_{1},\ldots
,\lambda_{d}}$ where  d is the co-dimension of the face in the
polytope. The elements of the  local  group $ G_F$ are of the form
$g =diag(e^{2\pi \sqrt{-1} \alpha_{1}},\ldots,e^{2\pi \sqrt{-1}
\alpha_{d}}),$  where $ \sum _{i=1}^{d}\alpha_{i}\lambda_{i} \in
N $, and $ 0 \leq \alpha_{i} <1 $. Recall that the age
\begin{equation}
\iota(g) = \alpha_{1} + \ldots + \alpha_{d}.
\end{equation}
is integral in quasi-$SL$ case  by definition $\ref{quasisl}$.

The singularity along the normal bundle of the sub-orbifold corresponding to interior of $F$ is of the form $\CC^d/G_F$.
These singularities  are same as Gorenstein {\bf toric} quotient singularities in complex  algebraic  geometry. This means they 
are toric (coming from a cone) {\bf$SL$ orbifold  singularity}($SL$ means linearization of a 
group element is $SL$,which in our case implies $\iota(g)$ is integral). 
Now let $N_{w}$ be the  lattice formed by
 $\{\lambda_{1},\ldots ,\lambda_{n}\}$, the characteristic vectors at a vertex $w$  contained in the face $F$. Let $m_{w}$  be the element in the dual lattice of
$N_w$ such that its  evaluation on each $\lambda_{i}$ is one.
 Now from  Lemma 9.2  of \cite{[CP]} we know that the cone $ C_{w}$ contains an integral basis, say $e_1,\ldots,e_n$. Suppose $e_i= \sum a_{ij} \lambda_j$. By \eqref{boxf}
$e_i$ corresponds to an element of $G_w$,
and since the singularity is qausi-$SL$, $\sum a_{ij}$ is integral.
Hence $m_{w}$ evaluated on each $e_j$ is integral. So $m_{w}$ an element of the dual
of the  integral lattice $ N$.

 Consider the $(n-1)$-dimensional lattice polyhedron  $\Delta_w$ defined as  $\{x \in C_w\mid \<$ x $,m_w \>= 1\}$. Note that $\Delta_w =\{ \sum_{i=1}^n a_i \lambda_i \mid a_i \ge 0, \; \sum a_i = 1 \}$.
  For any face $F$  containing $w$  we define $\Delta_F = \Delta_w \cap C_F$.
	If $\{\lambda_i, \ldots, \lambda_d \} $ denote the characteristic set of $F$, then
	$\Delta_F =\{ \sum_{i=1}^d a_i \lambda_i \mid a_i \ge 0, \; \sum a_i = 1 \} $.
	Hence $\Delta_F$ is independent of the choice of $w$.



 \begin{remark}
 An element  $g \in G$ of an $SL$ orbifold  singularity can be diagonalized to the form g$=diag(e^{2\pi \sqrt{-1} \alpha_{1}},\ldots,e^{2\pi \sqrt{-1} \alpha_{d}}),$ where $0 \leq  \alpha_i < 1$ and
$\iota(g)= \alpha_{1} + \ldots +\alpha_{d}$ is integral.
\end{remark}

We make some definitions following \cite{[BT]}.
\begin{defn}
Let $G$ be a finite subgroup of $SL(d,\mathbb{C})$.  Denote  by $ \psi_ {i}(G)$ the
number of the conjugacy classes of $G$ having  $\iota(g)=i$. Define
\begin{equation}\label{lat1}
 W(G;uv)=\psi_ {0}(G)  + \psi_ {1}(G)uv + \ldots +\psi_ {d-1}(G) {(uv)} ^{d-1}.
 \end{equation}
 \end{defn}

\begin{defn}\label{ht}
 We   define height(g) = rank(g-I).
\end{defn}

\begin{defn}\label{ht_1}
  Let $G$ be a finite subgroup of $SL(d,\mathbb{C})$. Denote by $\tilde{\psi_ {i}}(G)$
   the number of the conjugacy classes of $G$ having the $height =d$  and $\iota(g)=i$.
 \begin{equation}\label{lat2}
 \widetilde{W}(G;uv)=\tilde{\psi}_{0}(G)  + \tilde{\psi}_{1}(G)uv + \ldots  + \tilde{\psi}_{d-1}(G) {(uv)} ^{d-1} .
 \end{equation}
 \end{defn}

\begin{defn}
For a lattice polyhedron  $\Delta_F$ defining a $SL$ singularity $\CC^d/ G_F$, we define the following:
\begin{equation}
W(\Delta_F; uv) =W(G_F; uv) .
\end{equation}
\begin{equation}\label{psi}
\psi_ {i}(\Delta_F)=\psi_ {i}(G_F) .
\end{equation}
\begin{equation}\label{psi1}
\widetilde{W}(\Delta_F; uv) =\widetilde{W}(G_F; uv) .
\end{equation}
\begin{equation}\label{tpsi}
\tilde{\psi}_{i}(\Delta_F)=\tilde{\psi}_{i}(G_F) .
\end{equation}
\end{defn}
\subsection{E-polynomial for quastiotoric orbifold}
 \begin{defn}
  We define the $E$-polynomial of a quasitoric orbifold  ${\bf X}$ as follows
   \begin{equation}
    E_{quas}({\bf X }:u,v)= \Sigma _{p,q} (-1)^{p+q}  h^{p,q}({\bf X}) u^{p}v^{q} .
   \end{equation}

 \end{defn}
If $ X_i$ is the stratification of the the quasitoric orbifold by inverse image of the quotient map on  interior of 
faces $F_i$.{\bf Here $i$ runs over all the faces}.
\begin{theorem}
 \begin{equation}
   E_{quas}({\bf X }:u,v)=\Sigma _i E(X_i:u,v) .
 \end{equation}
\end{theorem}
\begin{proof}
 Let $X^{'}$ be  a projective toric orbifold whose Hodge structure has been pulled backed to ${\bf X}$. The 
by proposition $3.4$ \cite{[BT]} we have
\begin{equation}
  E_{quas}({\bf X }:u,v) =E(X^{'}:u,v)=\Sigma _i E(X^{'} _i:u,v).
\end{equation}
Where is $X{'}_i$ is stratification by inverse images of interiors of faces of the polytope of $X^{'}$.
 Since the two orbifolds have combinatorially equivalent polytopes number of  faces of a given dimension is same.
And since the stratas are algebraic tori  of dimension equal to  its corresponding face , we can replace $X{'}_i$ by
the corresponding $X_i$ in the right  hand most sum. The identification of $X{'}_i$ with $X_i$ is by the combinatorial
equivalence map.
\end{proof}
\begin{defn}
We define
\begin{equation}\label{orb} 
E_{orb}({\bf X }:u,v)= \Sigma _{p,q} (-1)^{p+q} h^{p,q}_{orb}({\bf X}) u^{p}v^{q}.
\end{equation}
\end{defn}
From the above discussions and since each $G_F$ is Abelian, it is easy to prove
\begin{equation}\label{orbpol}
E_{orb}({\bf X }:u,v)=  \Sigma_i E_{quas}( \bar{X_{i}}:u,v) \widetilde{W}(\Delta_{F_{i}},uv).
\end{equation}
 The following  can also be seen from what has been discussed  in the  previous subsection 
\begin{equation}\label{morestrat}
 W(\Delta_{F_{i}},uv)= \Sigma_{X_ {j} \geq  X_ {i}}  \widetilde{W}(\Delta_{F_{j}},uv).
 \end{equation}

\begin{equation}\label{newpon}
 E_{orb}({\bf X }:u,v)=\Sigma _{i}E(X_{i},u,v)W(\Delta_{F_{i}},uv).
\end{equation}

where $X_ {j} \geq  X_ {i}$ means $X_i \subset \bar{X_{j}}$ and {\bf X} is a quasi-$SL$ quasitoric orbifold.

We generalize $E_{st}$ defined in $6.7$  \cite{[BT]} to quasitorics as it has similar stratification in to $X_i's$
\begin{equation}\label{newpon}
 E_{st}({\bf X }:u,v)=\Sigma _{i}E(X_{i},u,v)W(\Delta_{F_{i}},uv).
\end{equation}

 Comparing our $E_{orb}$ with their $E_{st}$ we have.
 \begin{equation}\label{newpon_2}
E_{st}({\bf X }:u,v)=E_{orb}({\bf X }:u,v)
\end{equation}


\subsection{Proof of the main theorem}
We state the theorem again  for the reader's convenience.
\begin{theorem}
  For crepant blowdowns(or blowups) orbifold  Hodge numbers of quasi-$SL$ quasitoric orbifold do not change. 
\end{theorem}
\begin{proof}
 Let $\rho: \hat{{\bf X}}\rightarrow { \bf X }$ be a  crepant blowdown of omnioriented quasi-$SL$ quasitoric orbifolds. We set $ \hat{X_{i}}:=\rho^{-1}(X_{i})$. 
Then $\hat{X_{i}}$ has a natural
 stratification 
 it is enough to prove
 \begin{equation}
  E_{st}(\hat{{\bf X}})=E_{st}({ \bf X })
 \end{equation}
 But since quasi-$SL$ quasitoric orbifold have Gorenstein torodial singularity defined in \cite{[BT]} and a blow up effects only singularity cone of the face which is blowed up  and 
 neighboring
 cones, where things are toric and since  no global patching is required , the proof of Batyrev-Dias can be imitated here.(see theorem 6.2 \cite{[BT]})
\end{proof}
\section{appendix}
 Hear we give an argument why every simple polytope is combinatorially equivalent to a rational polytope. Every vertex $v$ of a $n$-dimensional simple polytope $P $is the solution of $n$ 
 linear equations. The solution set of each equation are  hyperplanes containing the codimension one sets whose intersection is the vertex. Since the coefficients form a linearly 
 independent set
 ,we can perturb them to get a system of coordinates which are rational and linearly independent and also the terms which are not attached to a variable can be made rational such that
 the new  rational solution vertex and the corresponding hyperplanes forming the new  codimension one faces and the faces  of $P$ which are untouched, 
 form a polytope combinatorially equivalent to the original polytope $P$(since rationals are dense).
    Now
 do the same for the adjacent vertices not changing the codimension one faces already made rational. After dong this for adjacent vetrices of $v$ do it for vertices adjacent to these vertices.
   If at a stage we reach all  vertices adjacent to a level of vertices, have gone through this adjustment,    we conclude that all vertices have been adjusted. To see this  if 
   there was any vertex that has been left out
 we can connect  one of the adjusted rational vertex to this vertex by a path of edges and since the vertex of an edge is adjacent to the other vertex of 
 an edge, so this vertex would have received adjustment at some stage.
\section{acknowledgement}
I thank  Institute of Mathematical Sciences Chennai and Harish Chandra research institute for my Post doctoral grant under which the research activities
for this paper was carried out.



\begin{thebibliography}{[Dieu]}






\bibitem{[WB]}  Walter L. Baily, Jr  The Decomposition Theorem for V-Manifolds American Journal of Mathematics,
Vol. 78, No. 4 (1956), 862-888.


\bibitem{[BT]} V. V. Batyrev and D. I. Dais: Strong McKay correspondence, string-theoretic Hodge numbers and mirror symmetry,  Topology 35 (1996), no. 4, 901-929.



\bibitem{[BP]} V. M. Buchstaber and T. E. Panov : Torus actions and their applications in topology and combinatorics,
 University Lecture Series {\bf 24}, American Mathematical Society, Providence, RI, 2002.




\bibitem{[CR]} W. Chen and Y. Ruan:  A new cohomology theory of orbifold, Comm. Math. Phys. {\bf 248}
  (2004), no. 1, 1-31.

\bibitem{[CP]} C.-H. Cho and M. Poddar: Holomorphic orbidiscs and Lagrangian Floer cohomology of symplectic toric orbifolds, arXiv:1206.3994

\bibitem{[DJ]} M. W. Davis and T. Januszkiewicz: Convex polytopes, Coxeter orbifolds and torus actions,
 Duke Math. J. {\bf 62} (1991), no.2, 417-451.





\bibitem{[GP]} S. Ganguli and M. Poddar: Blowdowns and McKay
correspondence on four dimensional quasitoric orbifolds, 
Osaka J. Math. {\bf 50} (2013) No. 2, 397-415 .

\bibitem{[GP2]} S. Ganguli and M. Poddar: Almost complex structure, blowdowns and McKay
 correspondence in  quasitoric orbifolds,  
 Osaka J. Math. {\bf 50} (2013) No. 4, 977-1025 .







\bibitem{[LP]} E. Lupercio and M. Poddar: The global McKay-Ruan correspondence via motivic integration,





\bibitem{[Po]} M. Poddar: Orbifold cohomology group of toric varieties.
Orbifolds in mathematics and physics (Madison, WI, 2001), 223-231,
Contemp. Math., 310, Amer. Math. Soc., Providence, RI, 2002.

\bibitem{[Po1]} M.Poddar: Orbifold Hodge numbers of Calabi-Yau hypersurfaces  arxiv:01071552v.

\bibitem{[PS]} M. Poddar and S. Sarkar: On quasitoric orbifolds,
 Osaka J. Math. {\bf 47} (2010) No. 4,
 1055-1076.

\bibitem{[EZ]} E. Zaslow, Topological Orbifold Models and Quantum cohomology Rings, Commun.
Math. Phys., 156 (1993) 301-331.

 \bibitem{[LT]} E. Lerman and S. Tolman: Hamiltonian torus actions on symplectic orbifolds and toric
varieties, Trans. Amer. Math. Soc. 349 (1997), no. 10, 4201-4230.

\bibitem{[CS]} C.Peters and J.Steenbrink :Mixed Hodge structures Springer

\bibitem{[Yas]} T. Yasuda: Twisted jets, motivic measures and orbifold cohomology, Compos. Math. 140 (2004).



\bibitem{[Ga]} S.Ganguli :  Mckay corespondence in Quasitoric orbifolds  arXiv:1308.3949.


\bibitem{[Yas1]}  T.Yasuda:  Motivic integration over Deligne-Mumford stacks , Advances in Mathematics,(2006) 207(2), 707-761.

\end{thebibliography}
\end{document}